\documentclass[12pt, reqno, a4paper]{amsart}

\usepackage{ amssymb, amsmath, enumerate, amsfonts, amsthm, mathrsfs, url, bm, mathtools}

\usepackage{xcolor}  	
\usepackage{hyperref}
\hypersetup{
colorlinks,
   linkcolor={cyan!80!black},
   citecolor={cyan!80!black},
 urlcolor={cyan!80!black}
}

\usepackage{color}

\usepackage[margin=1in]{geometry}

\RequirePackage{doi}

\usepackage{amscd}
\usepackage{amsfonts}
\usepackage{float}
\usepackage{color}
\usepackage[
backend=biber,
style=alphabetic,
]{biblatex}
\usepackage{bookmark}

\usepackage{amssymb}

\newtheorem{theorem}{Theorem}[section]

\newtheorem{example}{Example}[section]

\theoremstyle{definition}

\theoremstyle{remark}
\newtheorem{remark}{Remark}[section]

\numberwithin{equation}{section}

\newcommand{\Mod}[1]{\ (\mathrm{mod}\ #1)}

\renewcommand{\Re}{\mathrm{Re}}

\renewcommand{\leq}{\leqslant}
\renewcommand{\geq}{\geqslant}

\begin{document}

\title{A note on the standard zero-free region for $L$-functions}


\author{}
\address{}
\curraddr{}
\email{}
\thanks{}

\author{Sun-Kai Leung}
\address{D\'epartement de math\'ematiques et de statistique\\
Universit\'e de Montr\'eal\\
CP 6128 succ. Centre-Ville\\
Montr\'eal, QC H3C 3J7\\
Canada}
\curraddr{}
\email{sun.kai.leung@umontreal.ca}
\thanks{}


\date{}

\dedicatory{}


\begin{abstract}
In this short note, we establish a standard zero-free region for a general class of $L$-functions for which their logarithms have coefficients with nonnegative real parts, which includes the Rankin--Selberg $L$-functions for unitary cuspidal automorphic representations. 
\end{abstract}

\maketitle


\section{Introduction}
In 1899, de la Vall\'{e}e Poussin proved a quantitative prime number theorem
\begin{align*}
\psi(x)=x+O( x\exp(-c_1\sqrt{\log x}) )
\end{align*}
for some absolute constant $c_1>0$ by establishing the classical zero-free region for the Riemann zeta function $\zeta(s)$ of the form
\begin{align*}
\sigma>1-\frac{c_2}{\log (|t|+3)}
\end{align*}
for $s=\sigma+it.$ This generalizes to a zero-free region for a wide class of $L$-functions $L(f,s)$ of the form
\begin{align*}
\sigma>1-\frac{c}{d^4\log (\mathfrak{q}(f)(|t|+3))},
\end{align*}
where $d$ is the degree and $\mathfrak{q}(f)$ is the analytic conductor at $s=0$, with the possible exception for a simple real zero, 
in which case $f$ is self-dual (see \cite[Theorem 5.10]{MR2061214}). 

However, in order to apply \cite[Theorem 5.10]{MR2061214}, the existence of the Rankin--Selberg $L$-functions $L(f \otimes f, s)$ and $L(f \otimes \overline{f}, s)$ are necessary. In particular, when $L(f,s)=L(s,\pi \times \tilde{\pi})$ for some unitary cuspidal automorphic representations $\pi$ with $\tilde{\pi}$ being the contragredient representation of $\pi$, the existence of the associated Rankin–Selberg $L$-functions has yet to be confirmed in general; nonetheless, this would follow from the Langlands functoriality conjecture.

Following the effectivization of the Langlands--Shahidi method by Sarnak \cite{sarnak}, Goldfeld and Li \cite{MR3878594} proved that for any cuspidal automorphic representation $\pi$ of $GL_n(\mathbb{A}_{\mathbb{Q}})$ which is unramified and tempered at every non-archimedean place outside a set of Dirichlet density zero, the Rankin–Selberg $L$-function $L(s,\pi \times \tilde{\pi})$ has no zeros in the region 
\begin{align*}
\sigma>1-\frac{c_{\pi}}{(\log (|t|+3))^5}
\end{align*}
for some constant $c_\pi>0$ depending on $\pi,$ provided that $|t| \geq 1.$ 

Extending upon their work, Humphries and Brumley \cite[Theorem 1.9]{MR3980284} employed sieve methods to improve the zero-free region of
Goldfeld--Li to 
\begin{align*}
\sigma>1-\frac{c_{\pi}}{\log (|t|+3)}
\end{align*}
while removing their restriction that $K = \mathbb{Q}$ and  $\pi$ is unramified at every place. The automorphic representation
$\pi,$ however, is still required to be tempered at every non-archimedean place outside a set of
Dirichlet density zero. Furthermore, their zero-free region concerns only the $t$-aspect. More recently, Humphries and Thorner \cite[Theorem 5.1]{MR4378080} proved an unconditional refinement with improved uniformity in $\pi.$

In this short note, we generalize their results to a general class of $L$-functions for which their logarithms have coefficients with nonnegative real parts. 

\section{Main result}

We begin with introducing a general class of $L$-functions.
Throughout the paper, let 
\begin{align*}
L(f,s):=\sum_{n=1}^{\infty} \frac{\lambda_f(n)}{n^s}
\end{align*}
be an $L$-function which is absolutely convergent and zero-free on $\Re(s)>1.$ 
We denote the gamma factor as 
\begin{align*}
\gamma(f,s):=\prod_{j=1}^d {\pi}^{-\frac{s+\kappa_j}{2}}
\Gamma\left( \frac{s+\kappa_j}{2} \right),
\end{align*}
where $d$ is the degree and $\{\kappa_j\}_{j=1}^d$ are the Langlands parameters satisfying $\Re(\kappa_j)>-1$ for $j=1,\ldots,d.$
The coefficients of the Dirichlet series 
\begin{align*}
\log L(f,s) = \sum_{n=1}^{\infty} \frac{\Lambda_f(n)}{n^s\log n}
\end{align*}
satisfy the bound $|\Lambda_f(n)| \leq dn\log n$ for $n \geq 1.$
The complete $L$-function is defined as 
\begin{align*}
\Lambda(f,s):=q(f)^{\frac{s}{2}}\gamma(f,s)L(f,s)
\end{align*}
with conductor $q(f),$ which admits an analytic continuation to $\mathbb{C}$ except possibly poles of finite order at $s=1.$ Moreover, it satisfies the function equation
\begin{align*}
\Lambda(f,s)=\epsilon_f \Lambda(f,1-s),
\end{align*}
where $\epsilon_f$ is the root number.\footnote{It is plausible to extend Theorem \ref{thm:main} to a wider class of Dirichlet series, as the functional equations do not play a significant role here (see Section \ref{sec:proof} for details).}

\begin{remark}
The $L$-functions considered here do not necessarily have an Euler product, as the multiplicative structure of integers is not involved in establishing a zero-free region within the half-plane $\Re(s) \leq 1.$ Nevertheless, the existence of the Euler product 
\begin{align*}
L(f,s)=\prod_{p} \prod_{j=1}^d \left( 1-\frac{\alpha_j(p)}{p^s} \right)^{-1}
\end{align*}
for $\Re(s)>1,$ where $\{\alpha_j(p)\}_{p,j}$ are the Satake parameters, guarantees that there are no zeros in the half-plane $\Re(s) > 1.$
\end{remark}

We now state the main result.

\begin{theorem} \label{thm:main}
Let $L(f,s)$ be an $L$-function in the class above. 
Suppose $\Re(\Lambda_f(n)) \geq 0$ for $n \geq 1$ and $L(f,s)$ has at most a simple pole at $s=1.$\footnote{Note that $L(f,1) \neq 0$ since by assumption $\Re \left(\Lambda_f(n) \right) \geq 0$ for $n \geq 1.$} Then there exists an (effectively computable) absolute constant $c>0$ such that $L(f,s)$ has no zeros in the region 
\begin{align*}
\sigma > 1-\frac{c}{\log \left(\mathfrak{q}(f)(|t|+3)^d\right)}
\end{align*}
for $s=\sigma+it,$ with the possible exception for a simple real zero $\beta_f<1,$
where $\mathfrak{q}(f)$ is the analytic conductor $\mathfrak{q}(f,s)$ at $s=0,$ i.e.
 \begin{align*}
\mathfrak{q}(f):=q(f)\mathfrak{q}_{\infty}(0)=
q(f)\prod_{j=1}^{d}(|\kappa_j|+3).
 \end{align*}
\end{theorem}


The following examples are immediate consequences.

\begin{example}[Dirichlet $L$-functions]
Let $\chi \Mod{q}$ be a nontrivial Dirichlet character. Then there exists an (effectively computable) absolute constant $c>0$ such that the Dirichlet $L$-function $L(s,\chi)$ has no zeros in the region
\begin{align*}
\sigma > 1-\frac{c}{\log \left( q (|t|+3) \right)}
\end{align*}
except possibly for a simple real zero $\beta_{\chi}<1$
(see \cite[p. 93]{MR1790423}) by applying Theorem \ref{thm:main} to the product $\zeta(s)L(s,\chi).$
\end{example}
\begin{example}[Dedekind zeta functions]
Let $K/\mathbb{Q}$ be a number field of degree $d$. Then there exists an (effectively computable) absolute constant $c>0$ such that the Dedekind zeta function $\zeta_K(s)$ has no zeros in the region
\begin{align*}
\sigma > 1-\frac{c}{\log \left( |d_K|(|t|+3)^d \right)}
\end{align*}
except possibly for a simple real zero $\beta_{K}<1$, where $d_K$ is the discriminant (see \cite[Lemma 8.1 \& Lemma 8.2]{MR0447191}).
\end{example}

\begin{example}[Rankin--Selberg $L$-functions] 
Let $\pi$ be a unitary cuspidal automorphic representation of $GL_n(\mathbb{A}_K).$
Then there exists an (effectively computable) absolute constant $c>0$ such that the Rankin--Selberg $L$-function $L(s,\pi \times \tilde{\pi})$ has no zeros in the region
\begin{align} \label{eq:rs}
\sigma > 1-\frac{c}{ \log \left(\mathfrak{q}(\pi \times \tilde{\pi})(|t|+3)^{d^2}\right) }
\end{align}
except possibly for a simple real zero $\beta_{\pi \times \tilde{\pi}}<1$ (see \cite[Theorem 5.1]{MR4378080}).
\end{example}



\begin{remark}
The nonnegativity of $\Lambda_{\pi \times \tilde{\pi}}$ is guaranteed by the local Langlands conjecture for $GL_d$ over $p$-adic fields, proved by Harris--Taylor (see \cite[Remark, p.138]{MR2061214}). Also, since
\begin{align*}
\mathfrak{q}(\pi \times \tilde{\pi}) \leq 
\mathfrak{q}(\pi)^{2d} 3^{d^2[K:\mathbb{Q}]}
\end{align*}
(see \cite[equation (5.11)]{MR2061214}), in practice, the zero-free region (\ref{eq:rs}) can be replaced by
\begin{align*}
\sigma>1-\frac{c}{d\log \left( \mathfrak{q}(\pi)(|t|+3)^{d[K:\mathbb{Q}]} \right)}.
\end{align*}

\end{remark}




\section{Proof of Theorem \ref{thm:main}} \label{sec:proof}

 The method of de la Vall\'{e}e Poussin remains applicable in the absence of the Rankin--Selberg $L$-function $L(f \otimes \overline{f},s)$, as long as $\Re(\Lambda_f(n))\geq 0$ for $n \geq 1$ (though one must be wary of what one assumes, as pointed out in \cite[Remark 1.2]{MR3980284}).
\begin{proof}[Proof of Theorem \ref{thm:main}]
Let $\rho_0=\beta+i\gamma$ be a zero of $L(f,s)$ with $\beta \in \left[ 3/4,1\right]$ and $\gamma \neq 0.$ We denote the squared de la Vall\'{e}e Poussin product by 
\begin{align*}
L(g,\sigma):= 
|L^3(f,\sigma)L^4(f,\sigma+i\gamma)L(f,\sigma+2i\gamma)|^2. 
\end{align*}

For $n \geq 1,$ we have
\begin{align*}
 \Lambda_g(n) &=
2\left(3+4\cos(\gamma\log n)+\cos(2\gamma\log n)\right)
\Lambda_f(n)\\
&= |1+n^{-i\gamma}|^4 \Lambda_f(n).
\end{align*}
Since $\Re \left(\Lambda_f(n) \right) \geq 0,$ taking the real part of both sides yields
$
\Re \left(  \Lambda_g(n)  \right) \geq 0.$

Let $\sigma>1.$ Then 
\begin{align} \label{eq:dlvp}
-\Re \frac{L'}{L}(g,\sigma)=-6 \Re \frac{L'}{L}(f,\sigma)-8\Re \frac{L'}{L}(f,\sigma+i\gamma)-2\Re  \frac{L'}{L}(f,\sigma+2i\gamma) \geq 0.
\end{align}
Recall \cite[Proposition 5.7]{MR2061214} that for any $s$ in the strip $\sigma \in [1,5/4],$ we have
\begin{align} \label{eq:ik}
\frac{L'}{L}(f,s)+\frac{r}{s}+\frac{r}{s-1}=
\sum_{\kappa_j\;:\;|s+\kappa_j| \leq 1}\frac{1}{s+\kappa_j}+
\sum_{\rho\;:\;|s-\rho|\leq 1}\frac{1}{s-\rho}+O\left( \log \mathfrak{q}(f,s) \right),
\end{align}
where $r=0$ or $1$ is the order of pole at $s=1,$ $\rho$'s are the zeros of $L(f,s)$ and $\mathfrak{q}(f,s)$ is the analytic conductor.
 Since $\Re(\kappa_j) > -1$ and $\Re(\rho) \leq 1$, we have both 
 \begin{align} \label{eq:pos}
\Re \left(\frac{1}{s+\kappa_j}\right), \quad \Re \left(\frac{1}{s-\rho}\right)>0.
 \end{align}
Taking the real part of both sides of (\ref{eq:ik}) at $s=\sigma, \sigma+i\gamma, \sigma+2i\gamma$ gives
\begin{align*}
\Re \frac{L'}{L}(f,\sigma) \geq -\frac{r}{\sigma-1}
+O\left( \log \mathfrak{q}(f,\sigma) \right),
\end{align*}
\begin{align*}
\Re \frac{L'}{L}(f,\sigma+i\gamma) \geq
-\frac{r(\sigma-1)}{(\sigma-1)^2+\gamma^2}+\frac{1}{\sigma-\beta}+O(\log \mathfrak{q}(f,\sigma+i\gamma)),
\end{align*}
and
\begin{align*}
\Re  \frac{L'}{L}(f,\sigma+2i\gamma) \geq 
-\frac{r(\sigma-1)}{(\sigma-1)^2+4\gamma^2}+
O(\log \mathfrak{q}(f,\sigma+2i\gamma))
\end{align*}
respectively. Since by definition $d<\log \mathfrak{q}(f)$ and
\begin{align*}
\mathfrak{q}(f,s) \leq \mathfrak{q}(f)(|s|+3)^d
\end{align*}
(see \cite[equation (5.8)]{MR2061214}), combining the previous three inequalities with (\ref{eq:dlvp}) yields
\begin{align*}
\frac{4}{\sigma-\beta} \leq \frac{3r}{\sigma-1}+
\frac{4r(\sigma-1)}{(\sigma-1)^2+\gamma^2}+
\frac{r(\sigma-1)}{(\sigma-1)^2+4\gamma^2}
+A \mathcal{L}.
\end{align*}
for some absolute constant $A>10,$ where $\mathcal{L}:=\log (\mathfrak{q}(f)(|\gamma|+3)^d ).$ 

Suppose for now $r=0$ or $|\gamma| > (5A\mathcal{L})^{-1}.$
Then,
substituting $\sigma=1+(10A\mathcal{L})^{-1}$ into the inequality followed by careful
rearrangement, we conclude that
\begin{align*}
\beta < 1-\frac{1}{400A\mathcal{L}}.
\end{align*}

Otherwise, if $r=1$ and $0< |\gamma| \leq  (5A\mathcal{L})^{-1},$ then $\overline{\rho_0}=\beta-i\gamma$ is also a zero of $L(f,s)$ satisfying $|\sigma-\overline{\rho_0}| \leq 1.$ Also, since by assumption $\Re(\Lambda_f(n)) \geq 0$ for $n \geq 1,$ it follows from (\ref{eq:ik}) and (\ref{eq:pos}) that
\begin{align*} 
-\frac{1}{\sigma-1}+\frac{2(\sigma-\beta)}{(\sigma-\beta)^2+\gamma^2}
-B  \mathcal{L}
\leq \Re \frac{L'}{L}(f,\sigma) \leq 0
\end{align*}
for some absolute constant $B>0.$ To simplify notation, note that we can replace both $A$ and $B$ by $C:=\max\{A,B \}.$ Similarly, substituting $\sigma=1+(4C\mathcal{L})^{-1}$ into the inequality followed by careful
rearrangement, we conclude that
\begin{align*}
\beta \leq 1-\frac{1}{60C\mathcal{L}}.
\end{align*}

Finally, let $\{\beta_j\}_{j=1}^J$ be real zeros of $L(f,s)$ in the segment  $\left[\frac{3}{4},1 \right).$ Since 
\begin{align*}
-\Re \frac{L'}{L}(f,\sigma) \geq 0,
\end{align*}
it follows again from (\ref{eq:ik}) and (\ref{eq:pos}) that
\begin{align*}
\sum_{j=1}^J\frac{1}{\sigma-\beta_j}<\frac{1}{\sigma-1}+O\left( \log \mathfrak{q}(f) \right).
\end{align*}
Let $c_2>0.$ Suppose there are $n \geq 0$ real zeros $\beta_j$'s in the segment $\left[1-c_2(\log \mathfrak{q}(f))^{-1},1 \right).$ Then, substituting $\sigma=1+2c_2(\log \mathfrak{q}(f))^{-1}$ into the inequality gives
\begin{align*}
\frac{n\log \mathfrak{q}(f)}{3c_2} 
< \frac{\log \mathfrak{q}(f)}{2c_2}+O(\log \mathfrak{q}(f)),
\end{align*}
which implies that
\begin{align*}
\frac{n}{3}<\frac{1}{2}+O(c_2).
\end{align*}
Therefore, if $c_2$ is sufficiently small, then $n \leq 1,$ i.e. there is at most one simple zero in the segment. Finally, the theorem follows from choosing
$c=\min \left\{\frac{1}{400C},c_2 \right\}.$
\end{proof}



\section*{Acknowledgements}
The author is grateful to Andrew Granville for his advice and encouragement. He would also like to thank Maksym Radziwiłł and Jesse Thorner for helpful discussions, Peter Humphries for pointing out relevant papers, as well as the anonymous referee for valuable suggestions.

\printbibliography

@article {MR3878594,
    AUTHOR = {Goldfeld, D. and Li, X.},
     TITLE = {A standard zero free region for {R}ankin-{S}elberg
              {$L$}-functions},
   JOURNAL = {Int. Math. Res. Not. IMRN},
  FJOURNAL = {International Mathematics Research Notices. IMRN},
      YEAR = {2018},
    NUMBER = {22},
     PAGES = {7067--7136},
%      ISSN = {1073-7928},
%   MRCLASS = {11M41 (11F03 11F30 11F66)},
%  MRNUMBER = {3878594},
%MRREVIEWER = {A. Perelli},
%       DOI = {10.1093/imrn/rnx087},
%       URL = {https://doi.org/10.1093/imrn/rnx087},
}

@article {MR3980284,
    AUTHOR = {Humphries, P. and Brumley, F.},
     TITLE = {Standard zero-free regions for {R}ankin-{S}elberg
              {$L$}-functions via sieve theory},
   JOURNAL = {Math. Z.},
  FJOURNAL = {Mathematische Zeitschrift},
    VOLUME = {292},
      YEAR = {2019},
    NUMBER = {3-4},
     PAGES = {1105--1122},
%      ISSN = {0025-5874},
%   MRCLASS = {11M26 (11F66 11N36)},
%  MRNUMBER = {3980284},
%MRREVIEWER = {A. Perelli},
 %      DOI = {10.1007/s00209-018-2136-8},
%       URL = {https://doi.org/10.1007/s00209-018-2136-8},
}

@incollection {sarnak,
    AUTHOR = {Sarnak, P.},
     TITLE = {Nonvanishing of {$L$}-functions on {$\Re(s)=1$}},
 BOOKTITLE = {Contributions to automorphic forms, geometry, and number
              theory},
     PAGES = {719--732},
 PUBLISHER = {Johns Hopkins Univ. Press, Baltimore, MD},
      YEAR = {2004},
 %     ISBN = {0-8018-7860-8},
 %  MRCLASS = {11M06},
 % MRNUMBER = {2058625},
%MRREVIEWER = {Valentin\ Blomer},
}

@book {MR2061214,
    AUTHOR = {Iwaniec, H. and Kowalski, E.},
     TITLE = {Analytic number theory},
    SERIES = {American Mathematical Society Colloquium Publications},
    VOLUME = {53},
 PUBLISHER = {American Mathematical Society, Providence, RI},
      YEAR = {2004},
     PAGES = {xii+615},
%      ISBN = {0-8218-3633-1},
%   MRCLASS = {11-02 (11Fxx 11Lxx 11Mxx 11Nxx)},
%  MRNUMBER = {2061214},
%MRREVIEWER = {K. Soundararajan},
%       DOI = {10.1090/coll/053},
%       URL = {https://doi.org/10.1090/coll/053},
}

@inproceedings {MR0447191,
    AUTHOR = {Lagarias, J. C. and Odlyzko, A. M.},
     TITLE = {Effective versions of the {C}hebotarev density theorem},
 BOOKTITLE = {Algebraic number fields: {$L$}-functions and {G}alois
              properties ({P}roc. {S}ympos., {U}niv. {D}urham, {D}urham,
              1975)},
     PAGES = {409--464},
 PUBLISHER = {Academic Press, London},
      YEAR = {1977},
%   MRCLASS = {12A75},
%  MRNUMBER = {0447191},
%MRREVIEWER = {Matti Jutila},
}

@article {MR4378080,
    AUTHOR = {Humphries, P. and Thorner, J.},
     TITLE = {Towards a {${\rm GL}_n$} variant of the {H}oheisel phenomenon},
   JOURNAL = {Trans. Amer. Math. Soc.},
%  FJOURNAL = {Transactions of the %American Mathematical Society},
    VOLUME = {375},
      YEAR = {2022},
    NUMBER = {3},
     PAGES = {1801--1824},
  %    ISSN = {0002-9947,1088-6850},
 %  MRCLASS = {11F66},
 % MRNUMBER = {4378080},
%MRREVIEWER = {Liyang\ Yang},
%       DOI = {10.1090/tran/8544},
% URL = %{https://doi.org/10.1090/tran/8544},
}

@book {MR1790423,
    AUTHOR = {Davenport, H.},
     TITLE = {Multiplicative number theory},
    SERIES = {Graduate Texts in Mathematics},
    VOLUME = {74},
   EDITION = {Third},
      NOTE = {Revised and with a preface by Hugh L. Montgomery},
 PUBLISHER = {Springer-Verlag, New York},
      YEAR = {2000},
     PAGES = {xiv+177},
%      ISBN = {0-387-95097-4},
%   MRCLASS = {11-02 (11-01 11Mxx %11Nxx)},
%  MRNUMBER = {1790423},
}

\end{document}